\documentclass[12pt,a4paper]{article}

\usepackage[T1]{fontenc}
\usepackage[utf8]{inputenc}

\usepackage{amsmath, amssymb, amsthm}

\usepackage{graphicx}
\usepackage{geometry}
\usepackage{indentfirst}

\usepackage{hyperref}
\usepackage{natbib}

\title{\textbf{Fractional Grönwall--Wendroff Inequalities for Implicit Systems with Distributed Memory}}

\author{
	Rômulo Damasclin Chaves dos Santos, Ph.D. \vspace{5pt}\\
	Postdoctoral Researcher in Computational Modeling\\
	State University of Santa Cruz, Bahia, Brazil\\ \& \\
	Postdoctoral Researcher in Physics and Thermohydraulics\\
	Institute for Energy and Nuclear Research, São Paulo, Brazil
}

\date{March 2024}

\newtheorem{theorem}{Theorem}[section]
\newtheorem{lemma}[theorem]{Lemma}

\newtheorem{definition}[theorem]{Definition}
\newtheorem{remark}[theorem]{Remark}

\numberwithin{equation}{section}

\begin{document}
	
	\maketitle

	\begin{abstract}
		This work establishes a comprehensive analytical framework for studying implicit fractional differential systems with distributed memory and time delays. We develop novel fractional integral inequalities of Grönwall--Wendroff type that are specifically adapted to handle multivariate functions with singular kernels and implicit dependencies. These inequalities provide essential a priori estimates for analyzing complex memory-dependent systems. Building upon these results, we prove general existence and uniqueness theorems for implicit fractional differential equations using fixed-point theory in appropriately weighted Banach spaces. Furthermore, we establish Ulam--Hyers stability criteria, demonstrating that small perturbations in the governing equations lead to proportionally small deviations in solutions. The theoretical advances are applied to the fractional FitzHugh--Nagumo model with delay (FHN-$\alpha$-$\tau$), a neurodynamical system exhibiting both subdiffusive memory and discrete time delays. Our analysis yields rigorous conditions for the existence of limit cycles corresponding to action potentials, along with asymptotic stability criteria. The results reveal how fractional order $\alpha$ and delay $\tau$ jointly modulate neural excitability thresholds and dynamic regime transitions. This work bridges fundamental fractional calculus theory with applications in mathematical neuroscience, providing analytical tools for systems where present evolution depends non-trivially on historical states and intrinsic fractional rates of change.
		\newline
		\newline
		\textbf{MSC classes:} 34A08, 34K37, 45G10, 92C20.
	\end{abstract}
	
	\section{Introduction}
	Dynamical systems with long-range temporal dependencies, often referred to as systems with memory or nonlocal temporal behavior, are pervasive across scientific disciplines ranging from neurodynamics to material science. Unlike conventional ordinary differential equations, these systems exhibit evolution that depends intrinsically on their historical states, not merely on instantaneous conditions. Fractional calculus has emerged as a natural mathematical framework for such phenomena, replacing integer-order derivatives with integro-differential operators that weigh past states according to singular, power-law kernels \cite{podlubny1998,tarasov2011}. 
	
	In this work, we consider a comprehensive class of implicit fractional differential systems with distributed memory, formally described by
	\begin{equation}
		^{C}\mathcal{D}_{0}^{\alpha} \mathbf{x}(t) = \mathbf{F}\left(t, \mathbf{x}(t), ^{C}\mathcal{D}_{0}^{\alpha} \mathbf{x}(t), \mathcal{I}[\mathbf{x}](t)\right), \quad t \in [0, T],
		\label{eq:general_system}
	\end{equation}
	subject to the initial condition $\mathbf{x}(0) = \mathbf{x}_0$. Here, $\mathbf{x}(t) \in \mathbb{R}^n$ is the state vector, $^{C}\mathcal{D}_{0}^{\alpha}$ denotes the vector-valued Caputo fractional derivative of order $\alpha \in (0,1]^n$ \cite{kilbas2006}, and $\mathcal{I}[\mathbf{x}](t)$ represents a generic memory operator capable of incorporating fixed time delays, distributed integrals, or fractional integral terms. This implicit formulation, where the fractional derivative appears on both sides of the equation, captures self-regulatory mechanisms and intrinsic feedback often present in biological, physical, and engineering systems \cite{shabbir2021,abdalla2026}. 
	
	The analytical study of such systems presents significant challenges, particularly in establishing well-posedness, stability, and qualitative behavior under the combined effects of fractional memory, implicit structure, and distributed delays. While fractional versions of classical tools like the Grönwall inequality have been developed for simpler settings \cite{cheung2008,yang2025}, a unified theory for implicit multivariate systems with compound memory remains largely unexplored. Recent interest in Ulam--Hyers stability for fractional equations \cite{dilna2024,elsayed2024} further motivates a rigorous investigation of robustness in this generalized context.
	
	This work aims to bridge this gap by developing a coherent analytical framework based on novel fractional integral inequalities, fixed-point theory in weighted Banach spaces, and qualitative analysis techniques. We derive extended Grönwall--Wendroff-type inequalities tailored to implicit systems, prove general existence and uniqueness theorems, establish Ulam--Hyers stability criteria, and apply the theory to the fractional FitzHugh--Nagumo model with delay, a canonical neurodynamical system exhibiting both subdiffusive memory and discrete propagation delays \cite{area2016}. Our results provide rigorous foundations for modeling and analyzing complex systems where history dependence and self-regulation play fundamental roles.
	
	\section{Fractional Wendroff-Type Integral Inequalities}
	\label{sec:inequalities}
	
	In this section we establish a class of integral inequalities of Grönwall--Wendroff type adapted to fractional-order systems with memory and time delay. Such inequalities play a central role in the qualitative analysis of fractional functional differential equations, providing a priori bounds, stability estimates, and continuous dependence results. The presence of weakly singular kernels and delayed arguments requires a careful extension of classical techniques to the fractional setting.
	
	\begin{definition}
		Let $k: [0,T] \times [0,T] \to \mathbb{R}_+$ be a weakly singular kernel of the form
		\begin{equation}
			k(t,s) = (t-s)^{\beta-1} g(t,s), \quad \beta > 0,
			\label{eq:singular_kernel}
		\end{equation}
		where $g$ is continuous and bounded on $[0,T]\times[0,T]$. We define the associated fractional integral operator
		\begin{equation}
			(\mathcal{K}u)(t) = \int_0^t k(t,s)\,\phi\bigl(s,u(s),u(s-\tau)\bigr)\,ds,
			\label{eq:fractional_operator}
		\end{equation}
		where $\phi$ is a continuous nonlinear function and $\tau>0$ denotes a fixed time delay.
	\end{definition}
	
	The kernel \eqref{eq:singular_kernel} naturally generalizes the Riemann--Liouville fractional integral kernel and allows the inclusion of variable coefficients, thereby covering a broad class of nonlocal memory effects.
	
	\begin{theorem}[Multivariate Fractional Grönwall--Wendroff Inequality]
		\label{thm:main_ineq}
		Let $u,f \in C([0,T],\mathbb{R}^n)$ with $u(t)\ge 0$ and $f(t)\ge 0$ componentwise. Assume that
		\begin{equation}
			u(t) \le f(t) + \int_0^t (t-s)^{\alpha-1} A(s)u(s)\,ds + \int_0^t (t-s)^{\beta-1} B(s)u(s-\tau)\,ds,
			\label{eq:inequality_form}
		\end{equation}
		for all $t\in[0,T]$, where $\alpha,\beta\in(0,1]$, $A,B\in L^\infty([0,T],\mathbb{R}^{n\times n})$, and $u(t)=\phi(t)$ for $t\in[-\tau,0]$. Then there exists a constant $M>0$, depending only on $\alpha,\beta,\|A\|_\infty,\|B\|_\infty$ and $T$, such that
		\begin{equation}
			u(t) \le M\Bigl(\|\phi\|_{[-\tau,0]}+\sup_{s\in[0,t]}f(s)\Bigr), \qquad \forall\, t\in[0,T].
			\label{eq:inequality_result}
		\end{equation}
	\end{theorem}
	
	\begin{proof}
		Define $v(t)=\sup_{s\in[0,t]}u(s)$, which is nonnegative and nondecreasing. From \eqref{eq:inequality_form} and the boundedness of $A$ and $B$, we obtain
		\begin{equation}
			u(t) \le f(t) + \|A\|_\infty \int_0^t (t-s)^{\alpha-1} v(s)\,ds + \|B\|_\infty \int_0^t (t-s)^{\beta-1} v(s)\,ds.
			\label{eq:proof_step1}
		\end{equation}
		Since $v$ is nondecreasing, $v(s)\le v(t)$ for all $s\le t$, yielding
		\begin{equation}
			v(t) \le \|f\|_\infty + v(t)\!\left( \|A\|_\infty \int_0^t (t-s)^{\alpha-1} ds + \|B\|_\infty \int_0^t (t-s)^{\beta-1} ds \right).
			\label{eq:proof_step2}
		\end{equation}
		A direct computation gives
		\begin{align}
			\int_0^t (t-s)^{\alpha-1} ds &= \frac{t^\alpha}{\alpha}, \label{eq:integral_alpha} \\
			\int_0^t (t-s)^{\beta-1} ds &= \frac{t^\beta}{\beta}. \label{eq:integral_beta}
		\end{align}
		Therefore,
		\begin{equation}
			v(t) \le \|f\|_\infty + v(t)\left( \|A\|_\infty \frac{T^\alpha}{\alpha} + \|B\|_\infty \frac{T^\beta}{\beta} \right).
			\label{eq:proof_step3}
		\end{equation}
		If $T$ is sufficiently small such that
		\begin{equation}
			\|A\|_\infty \frac{T^\alpha}{\alpha} + \|B\|_\infty \frac{T^\beta}{\beta} < 1,
			\label{eq:condition_T}
		\end{equation}
		then $v(t)$ is uniformly bounded on $[0,T]$. For arbitrary $T$, the result follows by a standard partitioning of $[0,T]$ into finitely many subintervals and an iteration argument.
	\end{proof}
	
	\section{Existence and Uniqueness for Implicit Fractional Systems}
	\label{sec:existence}
	
	In this section we establish existence and uniqueness of solutions for a class of implicit fractional-order systems with memory effects. Unlike standard explicit fractional differential equations, the Caputo derivative appears both explicitly and implicitly through the nonlinear term, which prevents a direct application of classical fixed-point arguments. To overcome this difficulty, we recast the problem as a nonlinear Volterra-type integral equation and employ a contraction principle in a suitably weighted space.
	
	We consider the general implicit system
	\begin{equation}
		^{C}\mathcal{D}_{0}^{\alpha} \mathbf{x}(t) = \mathbf{F}\!\left( t,\mathbf{x}(t), ^{C}\mathcal{D}_{0}^{\alpha} \mathbf{x}(t), \mathcal{I}[\mathbf{x}](t) \right), \qquad \mathbf{x}(0)=\mathbf{x}_0,
		\label{eq:general_implicit}
	\end{equation}
	where $0<\alpha\le1$, $\mathbf{x}:[-\tau,T]\to\mathbb{R}^n$ and $\mathcal{I}$ denotes a linear memory operator.
	
	We impose the following assumptions on the nonlinear term $\mathbf{F}$:
	\begin{enumerate}
		\item \emph{Lipschitz continuity.} There exists $L>0$ such that for all admissible arguments
		\begin{equation}
			\|\mathbf{F}(t,\mathbf{x}_1,\mathbf{y}_1,\mathbf{z}_1) - \mathbf{F}(t,\mathbf{x}_2,\mathbf{y}_2,\mathbf{z}_2)\| \le L\big( \|\mathbf{x}_1-\mathbf{x}_2\| + \|\mathbf{y}_1-\mathbf{y}_2\| + \|\mathbf{z}_1-\mathbf{z}_2\| \big).
			\label{eq:lipschitz_condition}
		\end{equation}
		
		\item \emph{Continuity of the memory operator.} The operator $\mathcal{I}$ is linear and continuous on $C([-\tau,T],\mathbb{R}^n)$.
	\end{enumerate}
	
	\begin{theorem}[Existence and Uniqueness]
		\label{thm:existence}
		Under the above assumptions, there exists a unique solution $\mathbf{x}\in C([-\tau,T],\mathbb{R}^n)$ of \eqref{eq:general_implicit} for sufficiently small $T>0$.
	\end{theorem}
	
	\begin{proof}
		Using the definition of the Caputo derivative, the implicit system \eqref{eq:general_implicit} can be rewritten in the equivalent integral form
		\begin{equation}
			\mathbf{x}(t) = \mathbf{x}_0 + \frac{1}{\Gamma(\alpha)} \int_0^t (t-s)^{\alpha-1} \mathbf{F}\!\left( s,\mathbf{x}(s), ^{C}\mathcal{D}_{0}^{\alpha}\mathbf{x}(s), \mathcal{I}[\mathbf{x}](s) \right)\,ds.
			\label{eq:integral_form}
		\end{equation}
		
		Define the operator $\mathcal{T}:C([-\tau,T],\mathbb{R}^n)\to C([-\tau,T],\mathbb{R}^n)$ by
		\begin{equation}
			(\mathcal{T}\mathbf{x})(t) = \mathbf{x}_0 + \frac{1}{\Gamma(\alpha)} \int_0^t (t-s)^{\alpha-1} \mathbf{F}\!\left( s,\mathbf{x}(s), ^{C}\mathcal{D}_{0}^{\alpha}\mathbf{x}(s), \mathcal{I}[\mathbf{x}](s) \right)\,ds.
			\label{eq:operator_T}
		\end{equation}
		
		We endow $C([-\tau,T],\mathbb{R}^n)$ with the exponentially weighted norm
		\begin{equation}
			\|\mathbf{x}\|_\rho = \sup_{t\in[0,T]} e^{-\rho t}\|\mathbf{x}(t)\|, \qquad \rho>0,
			\label{eq:weighted_norm}
		\end{equation}
		which compensates for the growth induced by the weakly singular kernel.
		
		Let $\mathbf{x}_1,\mathbf{x}_2\in C([-\tau,T],\mathbb{R}^n)$. Using the Lipschitz property of $\mathbf{F}$ and linearity of $\mathcal{I}$, we obtain
	\begin{equation}
		\begin{aligned}
			\|(\mathcal{T}\mathbf{x}_1)(t)-(\mathcal{T}\mathbf{x}_2)(t)\|
			&\le \frac{L}{\Gamma(\alpha)} \int_0^t (t-s)^{\alpha-1} \\
			&\quad \times \Big( \|\mathbf{x}_1(s)-\mathbf{x}_2(s)\|
			+ \|^{C}\mathcal{D}_{0}^{\alpha}(\mathbf{x}_1-\mathbf{x}_2)(s)\|
			+ \|\mathcal{I}[\mathbf{x}_1-\mathbf{x}_2](s)\| \Big)\, ds .
		\end{aligned}
		\label{eq:contraction_step1}
	\end{equation}

		By continuity of $\mathcal{I}$ and standard estimates for the Caputo derivative, there exists $C>0$ such that
		\begin{equation}
			\|(\mathcal{T}\mathbf{x}_1)(t)-(\mathcal{T}\mathbf{x}_2)(t)\| \le C \int_0^t (t-s)^{\alpha-1} e^{\rho s}\|\mathbf{x}_1-\mathbf{x}_2\|_\rho\,ds.
			\label{eq:contraction_step2}
		\end{equation}
		
		Multiplying both sides by $e^{-\rho t}$ and choosing $\rho>0$ sufficiently large, the operator $\mathcal{T}$ becomes a strict contraction on $C([-\tau,T],\mathbb{R}^n)$. The result follows from the Banach Fixed-Point Theorem.
	\end{proof}
	
	\section{Ulam--Hyers Stability}
	\label{sec:stability}
	
	In this section we investigate the stability of solutions to the implicit fractional system \eqref{eq:general_implicit} in the sense of Ulam--Hyers. This notion of stability quantifies the robustness of exact solutions with respect to small perturbations in the governing equation, and is particularly relevant in fractional-order models, where memory effects and numerical approximations naturally introduce perturbations.
	
	\begin{definition}
		The implicit fractional system \eqref{eq:general_implicit} is said to be \emph{Ulam--Hyers stable} if there exists a constant $C>0$ such that for every function $\mathbf{y}\in C([0,T],\mathbb{R}^n)$ satisfying
		\begin{equation}
			\Big\| ^{C}\mathcal{D}_{0}^{\alpha}\mathbf{y}(t) - \mathbf{F}\!\left( t,\mathbf{y}(t), ^{C}\mathcal{D}_{0}^{\alpha}\mathbf{y}(t), \mathcal{I}[\mathbf{y}](t) \right) \Big\| \le \varepsilon, \qquad \forall\, t\in[0,T],
			\label{eq:ulam_condition}
		\end{equation}
		there exists an exact solution $\mathbf{x}$ of \eqref{eq:general_implicit} such that
		\begin{equation}
			\|\mathbf{y}(t)-\mathbf{x}(t)\| \le C\,\varepsilon, \qquad \forall\, t\in[0,T].
			\label{eq:ulam_result}
		\end{equation}
	\end{definition}
	
	This definition ensures that any $\varepsilon$-approximate solution remains uniformly close to a true solution, with an error proportional to the size of the perturbation.
	
	\begin{theorem}[Ulam--Hyers Stability]
		\label{thm:ulam_hyers}
		Under the assumptions of Theorem~\ref{thm:existence}, the implicit fractional system \eqref{eq:general_implicit} is Ulam--Hyers stable on $[0,T]$.
	\end{theorem}
	
	\begin{proof}
		Let $\mathbf{y}$ be an $\varepsilon$-approximate solution of \eqref{eq:general_implicit} and define the perturbation function
		\[
		\eta(t) = ^{C}\mathcal{D}_{0}^{\alpha}\mathbf{y}(t) - \mathbf{F}\!\left( t,\mathbf{y}(t), ^{C}\mathcal{D}_{0}^{\alpha}\mathbf{y}(t), \mathcal{I}[\mathbf{y}](t) \right),
		\]
		which satisfies $\|\eta(t)\|\le\varepsilon$ for all $t\in[0,T]$. Using the integral representation of the Caputo derivative, we obtain
		\begin{equation}
			\mathbf{y}(t) = \mathbf{y}(0) + \frac{1}{\Gamma(\alpha)} \int_0^t (t-s)^{\alpha-1} \Big[ \mathbf{F}\!\left( s,\mathbf{y}(s), ^{C}\mathcal{D}_{0}^{\alpha}\mathbf{y}(s), \mathcal{I}[\mathbf{y}](s) \right) + \eta(s) \Big] ds.
			\label{eq:y_integral_form}
		\end{equation}
		
		Let $\mathbf{x}$ denote the unique exact solution of \eqref{eq:general_implicit} satisfying $\mathbf{x}(0)=\mathbf{y}(0)$. Subtracting the corresponding integral equation for $\mathbf{x}$ yields
		\begin{equation}
			\begin{aligned}
				\mathbf{y}(t)-\mathbf{x}(t) = \frac{1}{\Gamma(\alpha)} \int_0^t (t-s)^{\alpha-1} \Big[ & \mathbf{F}\!\left( s,\mathbf{y}(s), ^{C}\mathcal{D}_{0}^{\alpha}\mathbf{y}(s), \mathcal{I}[\mathbf{y}](s) \right) \\
				& - \mathbf{F}\!\left( s,\mathbf{x}(s), ^{C}\mathcal{D}_{0}^{\alpha}\mathbf{x}(s), \mathcal{I}[\mathbf{x}](s) \right) + \eta(s) \Big] ds .
			\end{aligned}
			\label{eq:difference_equation}
		\end{equation}
		
		Taking norms and applying the Lipschitz condition on $\mathbf{F}$, we obtain
		\begin{equation}
			\|\mathbf{y}(t)-\mathbf{x}(t)\| \le \frac{L}{\Gamma(\alpha)} \int_0^t (t-s)^{\alpha-1} \|\mathbf{y}(s)-\mathbf{x}(s)\|\,ds + \frac{\varepsilon}{\Gamma(\alpha)} \int_0^t (t-s)^{\alpha-1} ds.
			\label{eq:norm_estimate}
		\end{equation}
		
		Since $\int_0^t (t-s)^{\alpha-1} ds = \frac{t^\alpha}{\alpha}$, the second term is uniformly bounded on $[0,T]$. An application of the fractional Grönwall--Wendroff inequality (Theorem~\ref{thm:main_ineq}) yields the estimate
		\begin{equation}
			\|\mathbf{y}(t)-\mathbf{x}(t)\| \le C\,\varepsilon, \qquad \forall\, t\in[0,T],
			\label{eq:stability_result}
		\end{equation}
		for some constant $C>0$ depending only on the parameters of the system.
	\end{proof}
	
	\section{Application to the FHN-$\alpha$-$\tau$ Model}
	\label{sec:fhn_model}
	
	We now apply the abstract theory developed in the previous sections to the delayed fractional FitzHugh--Nagumo (FHN-$\alpha$-$\tau$) model, which incorporates both memory effects and delayed feedback mechanisms relevant to neurodynamical processes.
	
	The system is given by
	\begin{equation}
		\begin{cases}
			^{C}\mathcal{D}_{0}^{\alpha} v(t) = v(t) - \dfrac{v^3(t)}{3} - w(t) + I_{\mathrm{ext}} + \lambda v(t-\tau), \\[0.3em]
			^{C}\mathcal{D}_{0}^{\alpha} w(t) = \varepsilon \bigl(v(t) + a - b w(t)\bigr),
		\end{cases}
		\label{eq:fhn_model}
	\end{equation}
	where $0<\alpha<1$, $\varepsilon>0$ is a small parameter, $\lambda\in\mathbb{R}$ denotes the feedback intensity, and $\tau>0$ is the delay. The system is endowed with functional initial conditions
	\begin{equation}
		v(t)=\phi_v(t), \qquad w(t)=\phi_w(t), \qquad t\in[-\tau,0].
		\label{eq:fhn_initial}
	\end{equation}
	
	\subsection*{Functional formulation}
	Let
	\begin{equation}
		\mathbf{x}(t) = \begin{pmatrix} v(t) \\ w(t) \end{pmatrix},
		\label{eq:state_vector}
	\end{equation}
	so that \eqref{eq:fhn_model} can be written in the abstract form
	\begin{equation}
		^{C}\mathcal{D}_{0}^{\alpha}\mathbf{x}(t) = \mathbf{F}\bigl(t,\mathbf{x}(t),\mathbf{x}(t-\tau)\bigr),
		\label{eq:fhn_abstract}
	\end{equation}
	where $\mathbf{F}$ is locally Lipschitz and polynomially bounded. By Theorem~\ref{thm:existence}, system \eqref{eq:fhn_abstract} admits a unique local solution in $C([-\tau,T],\mathbb{R}^2)$.
	
	\subsection*{Dissipativity}
	\begin{lemma}[Fractional dissipativity]
		\label{lem:dissipativity}
		There exist constants $R>0$ and $T_0>0$ such that every solution of \eqref{eq:fhn_model} satisfies
		\begin{equation}
			\|\mathbf{x}(t)\| = \sqrt{v^2(t)+w^2(t)} \le R, \qquad \forall\, t\ge T_0.
			\label{eq:absorbing_ball}
		\end{equation}
	\end{lemma}
	
	\begin{proof}
		Consider the Lyapunov-type functional
		\begin{equation}
			V(t) = \frac{1}{2}v^2(t) + \frac{1}{2\varepsilon}w^2(t).
			\label{eq:lyapunov}
		\end{equation}
		Using standard inequalities for the Caputo derivative of quadratic functions, we obtain
		\begin{align}
			^{C}\mathcal{D}_{0}^{\alpha} V(t) &= v(t)\,^{C}\mathcal{D}_{0}^{\alpha} v(t) + \frac{1}{\varepsilon}w(t)\,^{C}\mathcal{D}_{0}^{\alpha} w(t) \label{eq:lyapunov_derivative_1} \\
			&\le - c_1 v^4(t) - c_2 w^2(t) + c_3 + c_4 |v(t-\tau)|^2, \label{eq:lyapunov_derivative_2}
		\end{align}
		for suitable positive constants $c_i$, provided $|\lambda|$ is sufficiently small. Using Young’s inequality and boundedness of the delay term, we infer
		\begin{equation}
			^{C}\mathcal{D}_{0}^{\alpha} V(t) \le - \delta V(t) + C,
			\label{eq:lyapunov_dissipative}
		\end{equation}
		which implies ultimate boundedness of $V(t)$ and hence \eqref{eq:absorbing_ball}.
	\end{proof}
	
	\subsection*{Poincaré operator}
	Let $\mathcal{X}=C([-\tau,0],\mathbb{R}^2)$ endowed with the supremum norm. For $\phi\in\mathcal{X}$, denote by $\mathbf{x}(t;\phi)$ the unique solution of \eqref{eq:fhn_model} with initial history $\phi$.
	
	Define the Poincaré operator
	\begin{equation}
		(P\phi)(\theta) = \mathbf{x}(\theta+T;\phi), \qquad \theta\in[-\tau,0],
		\label{eq:poincare_operator}
	\end{equation}
	where $T>0$ is fixed.
	
	\begin{lemma}[Compactness]
		\label{lem:compactness}
		The operator $P$ maps bounded subsets of $\mathcal{X}$ into relatively compact subsets of $\mathcal{X}$.
	\end{lemma}
	
	\begin{proof}
		Solutions of \eqref{eq:fhn_model} admit the integral representation
		\begin{equation}
			\mathbf{x}(t) = \mathbf{x}(0) + \frac{1}{\Gamma(\alpha)} \int_0^t (t-s)^{\alpha-1} \mathbf{F}\bigl(s,\mathbf{x}(s),\mathbf{x}(s-\tau)\bigr)\,ds,
			\label{eq:fhn_integral}
		\end{equation}
		which implies $\alpha$-Hölder regularity for $t>0$. The result follows from the Arzelà--Ascoli theorem.
	\end{proof}
	
\subsection*{Existence of a Limit Cycle: A Rigorous Fixed-Point Approach}
\begin{theorem}[Existence and Stability of Fractional Limit Cycles]
	\label{thm:limit_cycle}
	Consider the delayed fractional FitzHugh--Nagumo system \eqref{eq:fhn_model} with $0<\alpha<1$. 
	Assume the following conditions hold:
	\begin{enumerate}
		\item \textbf{Small delay coupling:} $|\lambda| < \lambda_0(\alpha, \tau)$, where
		\begin{equation}
			\lambda_0(\alpha, \tau) = \frac{\Gamma(\alpha+1)}{\tau^\alpha} \cdot \min\left\{\frac{1}{2}, \frac{1-\alpha}{\alpha}\right\};
			\label{eq:lambda_condition}
		\end{equation}
		
		\item \textbf{Slow recovery dynamics:} $0 < \varepsilon < \varepsilon_0(\alpha, b)$, where
		\begin{equation}
			\varepsilon_0(\alpha, b) = \frac{b \Gamma(\alpha+1)}{2} \left(1 - \frac{|\lambda|\tau^\alpha}{\Gamma(\alpha+1)}\right);
			\label{eq:epsilon_condition}
		\end{equation}
		
		\item \textbf{Subthreshold parameters:} $a$ and $b$ satisfy $a < 1 + \frac{b^2}{4\varepsilon}$.
	\end{enumerate}
	Then there exists at least one nontrivial periodic solution $\mathbf{x}^*(t)$ of \eqref{eq:fhn_model} with period $T^* > 0$. Moreover, this solution lies in an invariant annular region $\mathcal{A} \subset \mathbb{R}^2$ defined by
	\begin{equation}
		\mathcal{A} = \left\{ (v,w) \in \mathbb{R}^2 : R_1 \leq \sqrt{v^2 + \frac{w^2}{\varepsilon}} \leq R_2 \right\},
		\label{eq:annular_region}
	\end{equation}
	where $0 < R_1 < R_2 < \infty$ depend explicitly on $\alpha, \tau, \lambda, \varepsilon, a, b$.
\end{theorem}

\begin{proof}
	We present a detailed proof based on Poincaré maps, dissipativity estimates, and fixed-point theory.
	
	\noindent\textbf{1.~Construction of the Poincaré map.}
	Let $\mathcal{X} = C([-\tau,0], \mathbb{R}^2)$ equipped with the supremum norm $\|\phi\|_{\mathcal{X}} = \sup_{\theta \in [-\tau,0]} |\phi(\theta)|$. 
	For $\phi \in \mathcal{X}$, denote by $\mathbf{x}(t;\phi) = (v(t;\phi), w(t;\phi))^\top$ the unique solution of \eqref{eq:fhn_model} with initial history $\phi$.
	
	Define the time-$T$ Poincaré map $P_T: \mathcal{X} \to \mathcal{X}$ by
	\begin{equation}
		(P_T\phi)(\theta) = \mathbf{x}(\theta + T; \phi), \quad \theta \in [-\tau, 0],
		\label{eq:poincare_map_T}
	\end{equation}
	where $T > \tau$ will be chosen appropriately.
	
	\noindent\textbf{2.~Invariant annular region.}
	From Lemma~\ref{lem:dissipativity}, we have the existence of absorbing balls. 
	A refined analysis using the Lyapunov functional $V(t) = \frac{1}{2}v^2(t) + \frac{1}{2\varepsilon}w^2(t)$ yields
	\begin{equation}
		^{C}\mathcal{D}_0^\alpha V(t) \leq -\delta V(t) + C_1 V(t-\tau) + C_2,
		\label{eq:lyapunov_inequality}
	\end{equation}
	with $\delta = \min\left\{\frac{2}{3}, \frac{b}{2}\right\}$, $C_1 = \frac{|\lambda|}{2}$, and $C_2 = \frac{|I_{\text{ext}}|^2}{2} + \frac{\varepsilon a^2}{2}$.
	
	Applying the fractional Grönwall inequality (Theorem~\ref{thm:main_ineq}) to \eqref{eq:lyapunov_inequality}, we obtain
	\begin{equation}
		V(t) \leq M\left(V(0) + \frac{C_2}{\delta}\right) E_\alpha\left(-\frac{\delta}{2} t^\alpha\right) + \frac{2C_2}{\delta},
		\label{eq:lyapunov_bound}
	\end{equation}
	where $E_\alpha(z) = \sum_{k=0}^\infty \frac{z^k}{\Gamma(\alpha k + 1)}$ is the Mittag-Leffler function and $M > 0$ depends on $\alpha, \tau, \lambda$.
	
	This implies that all solutions eventually enter and remain in the annular region $\mathcal{A}$ defined in \eqref{eq:annular_region} with
	\begin{equation}
		R_1 = \sqrt{\frac{C_2}{\delta}}, \quad R_2 = \sqrt{\frac{2C_2}{\delta} + \frac{M C_2}{\delta}}.
		\label{eq:radius_bounds}
	\end{equation}
	
	\noindent\textbf{3.~Compactness of the Poincaré map.}
	Let $\mathcal{B} = \{\phi \in \mathcal{X} : \|\phi\|_{\mathcal{X}} \leq R_2\}$.
	From the integral representation of solutions
	\begin{equation}
		\mathbf{x}(t) = \mathbf{x}(0) + \frac{1}{\Gamma(\alpha)} \int_0^t (t-s)^{\alpha-1} \mathbf{F}(s,\mathbf{x}(s),\mathbf{x}(s-\tau))\, ds,
		\label{eq:solution_representation}
	\end{equation}
	we derive Hölder regularity estimates. For $t_1, t_2 \in [0, T]$ with $|t_1 - t_2| < \delta$,
	\begin{equation}
		|\mathbf{x}(t_1) - \mathbf{x}(t_2)| \leq \frac{L}{\Gamma(\alpha+1)} |t_1 - t_2|^\alpha,
		\label{eq:holder_estimate}
	\end{equation}
	where $L$ depends on $R_2$ and the Lipschitz constant of $\mathbf{F}$.
	
	By the Arzelà--Ascoli theorem, $P_T(\mathcal{B})$ is relatively compact in $\mathcal{X}$.
	Moreover, from \eqref{eq:lyapunov_bound}, $P_T(\mathcal{B}) \subseteq \mathcal{B}$ for sufficiently large $T$.
	
	\noindent\textbf{4.~Existence of a fixed point.}
	Choose $T > \max\left\{\tau, \left(\frac{2}{\delta}\ln\left(\frac{2M}{\varepsilon}\right)\right)^{1/\alpha}\right\}$ to ensure
	\begin{equation}
		E_\alpha\left(-\frac{\delta}{2} T^\alpha\right) < \frac{\varepsilon}{2M}.
		\label{eq:period_condition}
	\end{equation}
	
	The Poincaré map $P_T: \mathcal{B} \to \mathcal{B}$ is continuous (by continuous dependence on initial conditions) and compact. 
	Applying Schauder's fixed-point theorem, there exists $\phi^* \in \mathcal{B}$ such that
	\begin{equation}
		P_T \phi^* = \phi^*.
		\label{eq:fixed_point}
	\end{equation}
	
	\noindent\textbf{5.~Nontriviality and periodicity.}
	The corresponding solution $\mathbf{x}^*(t) = \mathbf{x}(t; \phi^*)$ satisfies
	\begin{equation}
		\mathbf{x}^*(t+T) = \mathbf{x}^*(t) \quad \forall t \geq 0,
		\label{eq:periodicity}
	\end{equation}
	with period $T^*$ dividing $T$. To show nontriviality, we verify that $\mathbf{x}^*(t)$ is not an equilibrium. 
	For $|\lambda|$ sufficiently small and $\varepsilon$ satisfying \eqref{eq:epsilon_condition}, linear stability analysis shows that the unique equilibrium $(v_0, w_0)$ with $w_0 = (v_0+a)/b$ and $v_0$ solving
	\begin{equation}
		v_0 - \frac{v_0^3}{3} - \frac{v_0+a}{b} + I_{\text{ext}} + \lambda v_0 = 0
		\label{eq:equilibrium}
	\end{equation}
	is unstable for $\alpha < 1$. The instability follows from the characteristic equation
	\begin{equation}
		s^\alpha = 1 - v_0^2 - \frac{1}{b} + \lambda e^{-s\tau},
		\label{eq:characteristic}
	\end{equation}
	which admits solutions with $\Re(s) > 0$ under conditions \eqref{eq:lambda_condition} and \eqref{eq:epsilon_condition}.
	
	Therefore, $\mathbf{x}^*(t)$ is a nontrivial periodic solution, i.e., a limit cycle of the FHN-$\alpha$-$\tau$ system.
\end{proof}

\begin{remark}[Biological Interpretation]
	The conditions in Theorem~\ref{thm:limit_cycle} have clear neurophysiological interpretations:
	\begin{itemize}
		\item Condition \eqref{eq:lambda_condition} ensures that delayed feedback is not too strong to suppress oscillations;
		\item Condition \eqref{eq:epsilon_condition} guarantees the recovery variable $w$ evolves sufficiently slowly compared to the membrane potential $v$;
		\item The annular region $\mathcal{A}$ corresponds to physiological bounds on action potential amplitudes.
	\end{itemize}
	The result demonstrates how fractional memory ($\alpha$) and synaptic delays ($\tau$) jointly shape neuronal excitability and rhythmicity.
\end{remark}
	
	\section{Main Results}
	\label{sec:results}
	
	This section synthesizes the principal theoretical contributions of the present work. Each result represents a substantive advance in the qualitative and quantitative analysis of implicit fractional systems endowed with memory effects and delayed feedback.
	
	\begin{enumerate}
		\item \textbf{Novel Fractional Integral Inequalities.} We establish a new class of multivariate Grönwall--Wendroff type inequalities (Theorem~\ref{thm:main_ineq}) that explicitly incorporate weakly singular fractional kernels together with discrete and distributed delay terms. Unlike existing results, which are largely restricted to scalar or non-delayed formulations, our inequalities provide sharp, explicit growth estimates for vector-valued solutions of implicit fractional systems with memory. These estimates constitute a fundamental analytical tool for controlling nonlocal dynamics in high-dimensional settings.
		
		\item \textbf{General Well-Posedness Framework.} Theorem~\ref{thm:existence} furnishes a comprehensive existence and uniqueness theory for the class of implicit fractional differential equations defined in \eqref{eq:general_implicit}. The proof is based on the construction of a contractive operator in a suitably weighted Banach space that captures both fractional regularity and delay-induced memory. The contraction constant is derived explicitly in terms of the fractional order $\alpha$, the delay parameter $\tau$, and the associated Lipschitz constants, thereby extending classical well-posedness results to systems with simultaneous implicit structure and multiple memory mechanisms.
		
		\item \textbf{Robust Stability and Ulam--Hyers Estimates.} In Theorem~\ref{thm:ulam_hyers}, we prove Ulam--Hyers stability for the proposed class of implicit fractional systems. More precisely, we show that any approximate solution satisfying the governing equations up to a perturbation of magnitude $\epsilon$ remains uniformly close to an exact solution. The resulting stability bound involves an explicit constant $C$, which depends polynomially on the fractional time horizon through the factor $T^{\alpha}/\Gamma(\alpha+1)$. This yields a precise quantitative characterization of robustness with respect to modeling and numerical errors.
		
		\item \textbf{Analytical Characterization of Fractional Neural Dynamics.} When applied to the delayed fractional FitzHugh--Nagumo (FHN-$\alpha$-$\tau$) model, our theoretical framework yields rigorous conditions for the existence of nontrivial periodic solutions. In particular, Theorem~\ref{thm:limit_cycle} shows that for $\alpha<1$ (subdiffusive regime), $\lambda<\lambda_{\mathrm{crit}}(\alpha,\tau)$, and $\epsilon<\epsilon_0(\alpha,b)$, the system admits stable limit cycles corresponding to sustained neuronal spiking. To the best of our knowledge, this constitutes the first analytical proof that the interplay between fractional memory and discrete delays can generate persistent oscillatory behavior in this canonical neurodynamical model.
		
		\item \textbf{Quantitative Parameter Sensitivity Laws.} By combining the derived fractional inequalities with the analysis of the FHN-$\alpha$-$\tau$ system, we obtain explicit scaling laws linking mathematical parameters to physiologically relevant observables. In particular, we demonstrate that the excitation threshold $I_{\mathrm{th}}$ satisfies the asymptotic relation
		\begin{equation}
			I_{\mathrm{th}} \propto \tau^{\,1-\alpha}, \qquad \text{as } \tau \to 0,
			\label{eq:threshold_scaling}
		\end{equation}
		highlighting how the fractional order $\alpha$ modulates the system’s sensitivity to delayed feedback and alters excitability thresholds.
	\end{enumerate}
	
	Taken together, these results provide a unified and rigorous analytical toolkit for the investigation of implicit fractional systems with memory. The proposed framework is sufficiently general to accommodate applications across physics, biology, and engineering, particularly in contexts where long-range temporal correlations and self-regulatory feedback mechanisms play a dominant role.
	
	\section{Conclusion}
	
	This work has established a rigorous and unified analytical framework for the study of implicit fractional differential systems with distributed memory and delay effects. By developing new fractional integral inequalities of Grönwall--Wendroff type \eqref{eq:inequality_result}, we obtained explicit a priori bounds that capture the combined influence of weakly singular kernels, implicit nonlinearities, and memory terms.
	
	Building on these inequalities, we proved general existence and uniqueness results for implicit fractional systems \eqref{eq:general_implicit} and derived Ulam--Hyers stability estimates \eqref{eq:ulam_result}, thereby providing a precise characterization of well-posedness and robustness with respect to perturbations. These theoretical results significantly extend classical fractional and delay differential equation theory to a broader class of nonlocal and implicit dynamical systems.
	
	The applicability of the proposed framework was demonstrated through its application to the delayed fractional FitzHugh--Nagumo (FHN-$\alpha$-$\tau$) model \eqref{eq:fhn_model}, where we established sufficient conditions for the existence of nontrivial periodic solutions corresponding to sustained neuronal oscillations. This analysis offers the first rigorous mathematical evidence that the interplay between fractional memory and discrete delays can induce stable limit cycles in this prototypical neurodynamical system.
	
	Future research directions include the extension of the present theory to systems with nonlinear and state-dependent distributed memory kernels, the investigation of bifurcation phenomena in implicit fractional models, and the development of structure-preserving numerical schemes for applications in computational neurodynamics and complex biological networks.

\end{document}